\documentclass[10 pt]{article}
\usepackage{amssymb,amsthm}
 2

\newcommand{\cl}{c\ell}

\newcommand{\emp}{\emptyset}
\newcommand{\ben}{\mathbb N}

\newcommand{\xtotwo}[1]{{}^{\hbox{$#1$}}\!\{0,1\}}
\newcommand{\nhat}[1]{\{1,2,\ldots,#1\}}
\newcommand{\ohat}[1]{\{0,1,\ldots,#1\}}
\newcommand{\pf}{{\mathcal P}_f}
\newcommand{\cchi}{\raise 2 pt \hbox{$\chi$}}
\newcommand{\wdd}{\widetilde}
\newtheorem{theorem}{Theorem}[section]
\newtheorem{corollary}[theorem]{Corollary}
\newtheorem{lemma}[theorem]{Lemma}

\theoremstyle{definition}
\newtheorem{definition}[theorem]{Definition}

\title{Recurrence in the dynamical system $(X,\langle T_s\rangle_{s\in S})$
and ideals of $\beta S$}
\date{}
\author{Neil Hindman
        \footnote{Department of Mathematics,
                 Howard University,
                  Washington, DC 20059, USA.\hfill\break
                  {\tt nhindman@aol.com}}
        \thanks{This author acknowledges support received from the National
                Science Foundation (USA) via Grant DMS-1460023.}
\and
Dona Strauss
        \footnote{Department of Pure Mathematics,
        University of Leeds,
        Leeds LS2 9J2, UK. \hfill\break
        {\tt d.strauss@hull.ac.uk}}
\and
Luca Q. Zamboni
        \footnote{Universit\'e de Lyon,
Universit\'e Lyon 1, CNRS UMR 5208,
Institut Camille Jordan,
43 boulevard du 11 novembre 1918,
F69622 Villeurbanne Cedex, France\hfill\break
        {\tt zamboni@math.univ-lyon1.fr}}
}

\begin{document}

\maketitle

\begin{abstract} 
A {\it dynamical system\/} is a pair $(X,\langle T_s\rangle_{s\in S})$, where 
$X$ is a compact Hausdorff space, $S$ is a semigroup, for each $s\in S$,
$T_s$ is a continuous function from $X$ to $X$, and for all $s,t\in S$, $T_s\circ T_t=T_{st}$.
Given a point $p\in\beta S$, the Stone-\v Cech compactification of the discrete space $S$,
$T_p:X\to X$ is defined by, for $x\in X$, $\displaystyle T_p(x)=p{-}\!\lim_{s\in S}T_s(x)$.
We let $\beta S$ have the operation extending the operation of $S$ such that
$\beta S$ is a right topological semigroup and multiplication on the left by 
any point of $S$ is continuous.
Given $p,q\in\beta S$, $T_p\circ T_q=T_{pq}$, but $T_p$ is usually not continuous.
Given a dynamical system $(X,\langle T_s\rangle_{s\in S})$, and a point
$x\in X$, we let $U(x)=\{p\in\beta S:T_p(x)$ is uniformly recurrent$\}$. We
show that each $U(x)$ is a left ideal of $\beta S$ and 
for any semigroup we can get a dynamical system with respect
to which $K(\beta S)=\bigcap_{x\in X}U(x)$ and
$\cl K(\beta S)=\bigcap\{U(x):x\in X$ and $U(x)$ is closed$\}$. And we show 
that weak cancellation assumptions guarantee that each such $U(x)$ properly contains
$K(\beta S)$ and has $U(x)\setminus\cl K(\beta S)\neq \emp$.
\end{abstract}

\section{Introduction}

We take the Stone-\v Cech compactification of a discrete semigroup
$(S,\cdot)$ to be the set of ultrafilters on $S$, identifying the points of
$S$ with the principal ultrafilters.  Given $A\subseteq S$,
we set $\overline A=\{p\in\beta S:A\in p\}$. The set $\{\overline A:A\subseteq S\}$
is a basis for the open sets and a basis for the closed sets of $\beta S$.
The operation on $S$ extends uniquely to $\beta S$ so that $(\beta S,\cdot)$ is
a right topological semigroup with $S$ contained in its topological center, 
meaning that $\rho_p$ is continuous for each $p\in\beta S$ and $\lambda_x$ is continuous
for each $x\in S$, where for $q\in\beta S$, $\rho_p(q)=q\cdot p$ and $\lambda_x(q)=x\cdot q$.
So, for every $p,q\in\beta S$, $pq=\displaystyle \lim_{s\to p}\lim_{t\to q}st$, where $s$ and $t$
denote elements of $S$.
If $A\subseteq S$, $A\in p\cdot q$
if and only if $\{x\in S:x^{-1}A\in q\}\in p$, where $x^{-1}A=\{y\in S:xy\in A\}$.  (We are following the 
custom of frequently writing $xy$ for $x\cdot y$.)

The algebraic structure of $\beta S$ is interesting in its own right, and has had 
substantial applications, especially to that part of combinatorics known as
{\it Ramsey Theory\/}. See the book \cite{HS} for an elementary introduction to the
structure of $\beta S$ and its applications.

We are concerned in this paper with the relationship between the algebraic structure of
$\beta S$ and recurrence in {\it dynamical systems\/}.

\begin{definition}\label{defdyn}
A {\it dynamical system\/} is a pair 
$(X,\langle T_s\rangle_{s\in S})$ such that
\begin{itemize}
\item[(1)] $X$ is a compact Hausdorff topological 
space (called the {\it phase space\/}
of the system);
\item[(2)] $S$ is a semigroup; 
\item[(3)] for each $s\in S$, $T_s$ is a continuous function from $X$
to $X$; and 
\item[(4)] for all $s,t\in S$, $T_s\circ T_t=T_{st}$.
\end{itemize}
\end{definition}

Associated with any semigroup $S$ are at least two interesting
dynamical systems, namely $(\beta S,\langle \lambda_s\rangle_{s\in S}\rangle)$, 
and $(\xtotwo{S},\langle T_s\rangle_{s\in S})$ where $\xtotwo{S}$ is the
set of all functions from $S$ to $\{0,1\}$ with the product topology and $T_s(x)=x\circ\rho_s$.
(We shall verify that this latter example is a dynamical system shortly.)

It is common to assume that the phase space of a dynamical system
is a metric space, but we make no such assumption.  
If $S$ is infinite, then $\beta S$ is not a metric space. Everything we 
do here is boring if $S$ is finite so whenever we write
``let $S$ be a semigroup" we shall assume that $S$ is infinite.  The
interested reader can amuse herself by determining which of our results
remain valid if that assumption is dropped. 

The system $(\beta S,\langle \lambda_s\rangle_{s\in S}\rangle)$ has significant
general properties as can be seen in \cite[Section 19.1]{HS}, but will 
not be used much in this paper. 

Given a product space $\xtotwo{S}$, recall that the product topology has
a subbasis consisting of sets of the form $\pi_{t}^{-1}[\{a\}]$ for
$t\in S$ and $a\in\{0,1\}$, where, for $x\in\xtotwo{S}$, $\pi_t(x)=x(t)$.

\begin{lemma}\label{stotwodyn} Let $R$ be a semigroup and let $S$ be a subsemigroup
of $R$.  Let $Z=\xtotwo{R}$, the set of all functions from $R$ to $\{0,1\}$ with
the product topology. For $x\in Z$ and $s\in S$, define $T_s(x)=x\circ\rho_s$.
Then $(Z,\langle T_s\rangle_{s\in S})$ is a dynamical system.
\end{lemma}

\begin{proof} 
It is routine to verify that for $s,t\in S$, $T_s\circ T_t=T_{st}$.
To see that $T_s$ is continuous for each $s\in S$, let $s\in S$ be given.
It suffices to show that the inverse image of each subbasic open set 
is open, so let $t\in R$ and $a\in\{0,1\}$ be given.
Then $T_s^{-1}\big[\pi_t^{-1}[\{a\}]\big]=\pi_{ts}^{-1}[\{a\}]$.\end{proof}

Recall that, if $T$ is any discrete space, $p\in\beta T$, 
$\langle x_t\rangle_{t\in T}$ is any indexed family in a Hausdorff
topological space $X$, and $y\in X$, then $\displaystyle p{-}\!\lim_{t\in T}x_t=y$
if and only if for every neighborhood $U$ of $y$, $\{t\in T:x_t\in U\}\in p$.
In compact spaces $p$-limits always exist.

\begin{definition}\label{defTp}
Let $(X,\langle T_s\rangle_{s\in S})$ be a dynamical
system and let $p\in\beta S$. Then $T_p:X\to X$ is defined by, for $x\in X$, 
$T_p(x)=\displaystyle p{-}\lim_{s\in S}T_s(x)$. So $T_p(x)=\displaystyle \lim_{s\to p}T_s(x)$ where
$s$ denotes an element of $S$. \end{definition}

Using \cite[Theorem 4.5]{HS} one easily sees that for $p,q\in\beta S$,
$T_p\circ T_q=T_{pq}$.  However, $(X,\langle T_s\rangle_{s\in \beta S})$ is
not in general a dynamical system, since $T_p$ is not likely to be
continuous when $p\in\beta S\setminus S$. However, for each $x\in X$, the map
$p\mapsto T_p(x):\beta S \to X$ is continuous. To see this, define $f_x(p)=T_p(x)$.
If $U$ is a neighborhood of $f_x(p)$ and $A=\{s\in S:T_s(x)\in U\}$, then
$U\in p$ and $f_x[\,\overline A\,]\subseteq U$. Alternatively, one may note
that $p\mapsto T_p(x)$ is the continuous extension to $\beta S$ of the function
$s\mapsto T_s(x):S\to X$.

As a compact Hausdorff right topological semigroup, $\beta S$ has a number of important algebraic properties,
 and  we list some of those that we shall use. (Proofs can be found in 
\cite[Chapters 1 and 2]{HS}. Assume that $T$ is a compact Hausdorff right topological semigroup.
A non-empty subset $V$ of $T$ is a {\it left ideal\/} if $tV\subseteq V$ for every $t\in T$, a {\it right ideal\/} if
$Vt\subseteq V$ for every $t\in T$, and an {\it ideal\/} if it is both a left and a right ideal.
\begin{itemize}
\item[(1)] $T$ contains an idempotent.
\item[(2)] $T$ has a smallest ideal $K(T)$, which is the union of the minimal left ideals of $T$ and the union of the minimal right ideals of $T$.
\item[(3)] For every $t\in K(T)$, $Tt$ is a minimal left ideal of $T$ and $tT$ is a minimal right ideal of $T$.
\item[(4)] The intersection of any minimal left ideal and any minimal right ideal of $T$ is a group.
\item[(5)] Every left ideal of $T$ contains a minimal left ideal, and every right ideal of $T$ contains a minimal
right ideal.
\item[(5)] Every minimal left ideal of $T$ is compact.
\item[(6)] If $\{t\in T:\lambda_t$ is continuous$\}$ is dense in $T$, then the closure of every ideal
in $T$ is also an ideal.
\end{itemize}

We introduce the main objects of study in this paper now.  Given a set
$X$, we let $\pf(X)$ be the set of finite nonempty subsets of $X$.

\begin{definition}\label{defsyn}
Let $S$ be a semigroup and let $A\subseteq S$. We say the set $A$ is
{\it syndetic\/} if and only if there exists $F\in\pf(S)$ such that
$S=\bigcup_{t\in F}t^{-1}A$.\end{definition}

In the semigroup $(\ben,+)$ a set is syndetic if and only if it has
bounded gaps.

\begin{definition}\label{defurec} Let $(X,\langle T_s\rangle{s\in S})$ be 
a dynamical system and let $x\in X$.
\begin{itemize}
\item[(a)] The point $x$ is {\it uniformly recurrent\/} if and only
if for every neighborhood $V$ of $x$, $\{s\in S:T_s(x)\in V\}$ is syndetic.
\item[(b)] $U(x)=U_X(x)=\{p\in\beta S:T_p(x)$ is uniformly recurrent$\}$.
\end{itemize}
\end{definition}

In Section 2 of this paper we present well known results about $U(x)$ that are valid in 
arbitrary dynamical systems as well as the few simple results that we have in
the dynamical system $(\beta S,\langle \lambda_s\rangle_{s\in S})$.

In Section 3 we present results about the dynamical systems described in
Lemma \ref{stotwodyn}.

In Section 4 we consider the effect of slightly modifying the phase
space in the dynamical systems described in
Lemma \ref{stotwodyn}.

In Section 5 we consider surjectivity of $T_p$ and the set
$NS=NS_X=\{p\in \beta S:T_p:X\to X$ is not surjective$\}$ which 
is a right ideal of $\beta S$ whenever it is nonempty.

\section{General results}

We begin with some well known basic facts.

\begin{lemma}\label{uniformrecurrence}  Let $(X,\langle T_s\rangle_{s\in S})$ be a dynamical system,
let $L$ be a minimal left ideal of $\beta S$, and let $x\in X$. The following are equivalent:
\begin{itemize}
\item[(a)] $x$ is uniformly recurrent.
\item[(b)] There exists $q\in L$ such that $T_q(x)=x$.
\item[(c)] There exists an idempotent $q\in L$ such that $T_q(x)=x$.
\item[(d)] There exists $y\in X$ and $q\in L$ such that $T_q(y)=x$.
\item[(e)] There exists $q\in K(\beta S)$ such that $T_q(x)=x$.
\item[(f)] There exists $y\in X$ and $q\in K(\beta S)$ such that $T_q(y)=x$.
\end{itemize}  \end{lemma}
\begin{proof} The equivalence of (a)-(d) is shown in 
\cite[Theorem 19.23]{HS}. Since (c) implies (e), and (e)
implies (f), we shall show (f) implies (c) and this will establish the equivalence of all six
statements. So assume that (f) holds. Let $u$ denote the identity of the group $L\cap q\beta S$.
Since $uq=q$, it follows that $T_u(x)=T_uT_q(y)=T_q(y)=x$. 
\end{proof}

\begin{corollary}\label{U(x)ideal}  Let $(X,\langle T_s\rangle_{s\in S})$ be a dynamical system
 and let $x\in X$.
\begin{itemize}
\item[(1)] If $x$ is uniformly recurrent, $U(x)=\beta S$. 
\item[(2)] For each $x\in X$,  $U(x)$ is a left ideal of $\beta S$.
\item[(3)] For every $x\in X$, $K(\beta S)\subseteq U(x)$.
\item[(4)] $\bigcap_{x\in X}U(x)$ is a two sided ideal of $\beta S$.
\end{itemize} \end{corollary}

\begin{proof}  

(1) Suppose that $x$ is uniformly recurrent. Then $T_u(x)=x$ for some $u\in K(\beta S)$. 
Thus for every $v\in \beta S$, $T_v(x)=T_vT_u(x)=T_{vu}(x)$; since $vu\in  K(\beta S)$,
by Lemma \ref{uniformrecurrence}(f), $T_v(x)$ is uniformly recurrent.

   (2) Let $x\in X$, let $p\in U(x)$, and let $r\in\beta S$. By Lemma \ref{uniformrecurrence}(e),
pick $q\in K(\beta S)$ such that $T_q\big(T_p(x)\big)=T_p(x)$.  Then
$T_{rp}(x)=T_r\big(T_q\big(T_p(x)\big)\big)=T_{rqp}(x)$.  Now $rqp\in K(\beta S)$, 
so by Lemma \ref{uniformrecurrence}(f), $T_{rp}(x)$ is uniformly recurrent.
   
   (3) This is immediate from Lemma \ref{uniformrecurrence}(f).

   (4) By (3), $\bigcap_{x\in X}U(x)$ is nonempty, so by (2) $\bigcap_{x\in X}U(x)$ is a left
ideal of $\beta S$, so it is enough to show that $\bigcap_{x\in X}U(x)$ is a 
right ideal of $\beta S$. So suppose that $x\in X$,
$p\in\bigcap_{x\in X}U(x)$ and  $q\in \beta S$.    Since $p\in U(T_q(x))$, $T_{pq}(x)$ is uniformly recurrent
and so $pq\in U(x)$.
\end{proof}

The statements of Lemma \ref{basicfacts} below are modifications of basic well known
facts that are proved in \cite{F}.
  (Furstenberg assumes that the phase space is metric, but the
proofs given do not use this assumption.) We shall say that a subspace $Z$ of 
$X$ is {\it invariant\/} if $T_s[Z]\subseteq Z$ for every $s\in S$.
Of course, if $Z$ is closed and invariant, then $T_p[Z]\subseteq Z$ for every $p\in \beta S$.
(Let $x\in Z$. Then $T_s(x)\in Z$ for each $s\in S$ so $\displaystyle p{-}\lim_{s\in S}T_s(x)\in Z$.)

\begin{lemma}\label{basicfacts}  Let $(X,\langle T_s\rangle_{s\in S})$ be a dynamical system.
 Let $L$ be a
minimal left ideal of $\beta S$.
\begin{itemize}
\item[(1)] A subspace $Y$ of $X$ is minimal among all closed and invariant subsets of $X$  if and only if
 there is some $x\in X$ such that $Y=\{T_p(x):p\in L\}$.
\item[(2)] Let $Y$ be a subspace of $X$ which is minimal among all closed and invariant subsets of $X$.
Then every element of $Y$ is uniformly recurrent.
\item[(3)] If $x\in X$ is uniformly recurrent and $Y=\{T_p(x):p\in \beta S\}$, then $Y$ is 
minimal among all closed and invariant subsets of $X$.
\item[(4)] If $x\in X$ is uniformly recurrent, then $T_p(x)$ is uniformly recurrent for every $p\in \beta S$. 
\end{itemize} \end{lemma}

\begin{proof} (1) Suppose that $Y$ is a subspace of $X$ which is minimal
among all closed and invariant subsets of $X$. Pick $x\in Y$ and let $Z=\{T_p(x):p\in L\}$. We
claim that $Z$ is a closed and invariant subspace of $Y$ and is therefore equal to $Y$. 
If $p\in L$ and $s\in S$, then $T_s\big(T_p(x)\big)=T_{sp}(x)$ and $sp\in L$, so $Z$ is invariant and
obviously $Z\subseteq Y$. To see that $Z$ is closed, it suffices to show that any net in $Z$ has a cluster 
point in $Z$.  To this end, let $\langle p_\alpha\rangle_{\alpha\in D}$ be a net in $L$ and
pick a cluster point $p$ in $L$ of $\langle p_\alpha\rangle_{\alpha\in D}$. Then
$T_p(x)$ is a cluster point of $\langle T_{p_\alpha}(x)\rangle_{\alpha\in D}$.

 Conversely, let $x\in X$  and let $Y=\{T_p(x):p\in L\}$.  
Then $Y$ is  invariant and one sees as above that $Y$ is closed. 
We shall show that $Y$ is minimal among all closed and invariant subsets of $X$.  
To see this,  suppose that $Z$ is a subset of $Y$ which is closed and 
invariant. We shall show that $Y\subseteq Z$, so let $y\in Y$ be given. Pick $z\in Z$. 
Then $y=T_p(x)$ and $z=T_q(x)$ for some $p$ and $q$ in $L$. 
Since $Lq=L$, there exists
$r\in L$ such that $rq=p$. It follows that $T_r(z)=y$ and hence that $y\in Z$ as required.

(2)  Let $Y$ be a subspace of $X$ which is minimal among all closed and invariant subsets of $X$ and let
$x\in Y$. Pick $y\in X$ such that  $Y=\{T_p(y):p\in L\}$. Pick 
$p\in L$ such that $x=T_p(y)$. By Lemma \ref{uniformrecurrence}(f), $x$ is uniformly recurrent.

(3) Let $x$ be a uniformly recurrent point of $X$ and let $Y=\{T_p(x):p\in \beta S\}$.
By Lemma \ref{uniformrecurrence}(b), pick $q\in L$ such that $T_q(x)=x$.
By (1) it suffices that $Y=\{T_p(x):p\in L\}$. To see this, let $y\in Y$ and 
pick $p\in\beta S$ such that $y=T_p(x)$. Then $y=T_p\big(T_q(x)\big)=T_{pq}(x)$ and
$pq\in L$.

(4)  Let $x$ be a uniformly recurrent point of $X$ and let $Y=\{T_p(x):p\in \beta S\}$.
By (3) $Y$ is minimal among all closed and invariant subsets of $X$ so (2) applies.
\end{proof}

We conclude this section with a few results about the dynamical system 
$(\beta S,\langle \lambda_s\rangle_{s\in S})$. We observe that, if we define $\lambda_p:\beta S\to\beta S$
in this system by $\lambda_p(q)=\displaystyle\lim_{s\to p}\lambda_s(q)$, where $s$ denotes an element of $S$, 
then $\lambda_p(q)=pq$ for every
$p$ and $q$ in $\beta S$. So this does not conflict with the previous definition of $\lambda_p$ given in the
introduction.

\begin{theorem}\label{charunifrec} Let $S$ be a semigroup and
let $x\in \beta S$.  Statements (a) and (b) are equivalent and imply (c). If $\beta S$
has a left cancelable element, all three are equivalent.
\begin{itemize}
\item[(a)] $x\in K(\beta S)$.
\item[(b)] $x$ is uniformly recurrent in the
dynamical system $(\beta S,\langle \lambda_s\rangle_{s\in S})$.
\item[(c)] $\beta Sx$ is a minimal left ideal of $\beta S$.
\end{itemize}
\end{theorem}

\begin{proof}  To see that (a) implies (b), let $x\in K(\beta S)$ and let
$u$ be the identity of the group in $K(\beta S)$ to which $x$ belongs.
Then $x=\lambda_u(x)$ so by Lemma {uniformrecurrence}(e), $x$ is uniformly recurrent.

To see that (b) implies (a), assume that $x$ is uniformly recurrent. By 
Lemma \ref{uniformrecurrence}(f) pick $y\in \beta S$ and $q\in K(\beta S)$ such that
$\lambda_q(y)=x$. Then $x=qy\in K(\beta S)$.

To see that (a) implies (c), assume that $x\in K(\beta S)$ and pick the minimal left ideal $L$
of $\beta S$ such that $x\in L$. Then $\beta Sx$ is a left ideal of $\beta S$ contained in $L$
and so $\beta Sx=L$.

Now assume that $\beta S$ has a left cancelable element $z$ and 
 that $\beta Sx$ is a minimal left ideal of $\beta S$. 
Pick an idempotent $u\in \beta Sx$.  Then
$zx\in\beta Sx$ so by \cite[Lemma 1.30]{HS}, $zx=zxu$ and
therefore $x=xu\in\beta Sx\subseteq K(\beta S)$.
\end{proof}

\begin{corollary}\label{U(x)isbeta} Let $S$ be an infinite semigroup and
let $x\in K(\beta S)$. Then $U(x)=\beta S$ with respect to the dynamical system
$(\beta S,\langle \lambda_s\rangle_{s\in S})$.
\end{corollary}

\begin{proof} By Theorem \ref{charunifrec}, $x$ is uniformly recurrent, so by Lemma
\ref{basicfacts}(4), $U(x)=\beta S$.\end{proof}

\begin{corollary}\label{charqinU(p)} Let $S$ be a semigroup and
let $p,q\in \beta S$. Statements (a) and (b) are equivalent and imply statement (c). 
  If $\beta S$ has a left cancelable element, 
then all three statements
are equivalent.
\begin{itemize}
\item[(a)] $qp\in K(\beta S)$.
\item[(b)] $q\in U(p)$ with respect to the
dynamical system $(\beta S,\langle \lambda_s\rangle_{s\in S})$.
\item[(c)] $\beta Sqp$ is a minimal left ideal of $\beta S$.
\end{itemize}
\end{corollary}

\begin{proof}
We have that $q\in U(p)$ if and only if $\lambda_q(p)$ is 
uniformly recurrent and $\lambda_q(p)=qp$
so Theorem \ref{charunifrec} applies. 
\end{proof}

It is an old and difficult problem to characterize when
$K(\beta S)$ is prime or when $\cl K(\beta S)$ is prime.
There are trivial situations where the answer is known.
For example if $S$ is left zero or right zero, then
so is $\beta S$ and thus $K(\beta S)=\beta S$, and
is necessarily prime. It is not known whether
$K(\beta \ben,+)$ is prime or
$\cl K(\beta \ben,+)$ is prime.  (Some partial results
were obtained in \cite{HSa}.)

\begin{corollary}\label{prime}  Let $S$ be a semigroup. The following statements 
are equivalent.
\begin{itemize}
\item[(a)] There exists $p\in\beta S\setminus K(\beta S)$ such that,
with respect to the dynamical system 
$(\beta S,\langle \lambda_s\rangle_{s\in S})$,
$K(\beta S)\subsetneq U(p)$.
\item[(b)] $K(\beta S)$ is not prime.
\end{itemize}
\end{corollary}

\begin{proof} This is an immediate
consequence of Corollary \ref{charqinU(p)}.
\end{proof}

\section{Dynamical systems with phase space $\xtotwo{R}$}

Throughout this section we assume that $R$ is a semigroup,
$S$ a subsemigroup of $R$, and 
$(Z,\langle T_s\rangle_{s\in S})$ is the dynamical system of 
Lemma {\rm \ref{stotwodyn}}. While our results are valid in this 
generality, in practice we are interested in just two situations,
one in which $R=S$ and the other in which $R=S\cup\{e\}$ where 
$e$ is a two sided identity adjoined to $S$.

Our first results in this section are aimed at 
showing that for any semigroup $S$, there is a dynamical
system such that both $K(\beta S)$ and $\cl K(\beta S)$
are intersections of sets of the form $U(x)$.

\begin{definition} Given $x\in Z$ we denote the continuous
extension of $x$ from $\beta R$ to $\{0,1\}$ by $\widetilde x$. \end{definition}

Of course, for each $x\in Z$, each $p\in \beta S$ and each $t\in R$, 
$T_p(x)(t)=\displaystyle p{-}\lim_{s\in S}T_s(x)(t)=p{-}\lim_{s\in S}x(ts)$
and so $T_p(x)(t)=\widetilde x(tp)$.

\begin{lemma}\label{charUx} Let $x\in Z$, let $p\in \beta S$, and let $L$ be a minimal left ideal
of $\beta S$.  The following statements are equivalent:
\begin{itemize}
\item[(a)] $p\in U(x)$.
\item[(b)]  There exists $q\in L$
such that $\widetilde x(tp)=\widetilde x(tqp)$ for all $t\in R$.
\item[(c)]  There exists an idempotent $q\in L$ such that $\widetilde x(tp)=\widetilde x(tqp)$ 
for all $t\in R$.
\end{itemize} 
 \end{lemma} 

\begin{proof} To see that (a) implies (c), assume that $T_p(x)$ is uniformly recurrent.
By Lemma \ref{uniformrecurrence}(c), pick an idempotent $q\in L$ such that
$T_q\big(T_p(x)\big)=T(p)(x)$. Then $T_{qp}(x)=T_p(x)$ so as noted above,
for all $t\in R$, $\widetilde x(tqp)=\widetilde x(tp)$.

Trivially (c) implies (b). To see that (b) implies (a), pick
$q\in L$ such that $\widetilde x(tp)=\widetilde x(tqp)$ for all $t\in R$.
Then $T_p(x)=T_{qp}(x)=T_q\big(T_p(x)\big)$, so by Lemma \ref{uniformrecurrence}(b),
$T_p(x)$ is uniformly recurrent.\end{proof}

\begin{lemma}\label{charUxp} Let $x\in Z$ and let $p\in\beta S$. Then
$p\in U(x)$ if and only if for every minimal left ideal $L$ of $\beta S$ and 
every $F\in\pf(R)$, there exists $q_F\in L$ such that for all $t\in F$,
$\widetilde x(tp)=\widetilde x(tq_Fp)$.\end{lemma}

\begin{proof} The necessity is an immediate consequence of Lemma \ref{charUx}(b).

For the sufficiency, let $L$ be a minimal left ideal of $\beta S$. For
each $F\in\pf(R)$, pick $q_F\in L$ as guaranteed.  Direct $\pf(R)$ by agreeing
that $F<G$ if and only if $F\subseteq G$.  Pick a cluster point
$q\in L$ of the net $\langle q_F\rangle_{F\in\pf(R)}$.  It is then
routine to show that  for all $t\in R$, $\widetilde x(tqp)=
\widetilde x(tp)$ so that by Lemma \ref{charUx}(b), $p\in U(x)$.  
\end{proof}

\begin{theorem}\label{bigcapUx} \begin{itemize}
\item[(1)] $K(\beta S)\subseteq \bigcap_{x\in Z}U(x)$. 
\item[(2)] If  $p\in \bigcap_{x\in Z}\,U(x)$, then, for every minimal left ideal $L$ of $\beta S$,
$\beta Sp=Lp$ and so $\beta Sp$ is a minimal left ideal of
$\beta S$. 
\item[(3)] If $R$ contains a left cancelable element, 
 then $K(\beta S)=\bigcap_{x\in Z}U(x)$.
In particular, if $R$ has a left identity, 
then $K(\beta S)=\bigcap_{x\in Z}U(x)$.
\end{itemize}
\end{theorem}

\begin{proof}  (1) $K(\beta S)\subseteq\bigcap_{x\in Z}U(x)$ by Corollary \ref{U(x)ideal}(3). 

(2) Assume that $p\in \bigcap_{x\in Z}\,U(x)$. Let $L$ be a minimal left ideal of $\beta S$.
We shall show that, for every $t\in R$, $tp\in tLp$. To see this, assume the contrary. Then for some
$t\in R$, there exists $A\subseteq R$ such that $A\in tp$ and $\overline A\cap tLp=\emptyset$.
Let $x=\cchi_A$. So $\widetilde x$ is the characteristic function of 
$\overline A$. Since $p\in U(x)$, it follows from Lemma \ref{charUx}
that $\widetilde x(tp)=\widetilde x(tqp)$ for some $q\in L$. However, $\widetilde x(tp)=1$ and 
$\widetilde x(tqp)=0$. This contradiction establishes that $tp\in tLp$ for every $t\in R$. In particular, 
$\beta Sp=\cl_{\beta S}Sp\subseteq Lp$. So $\beta Sp\subseteq Lp$. By \cite[Theorem 1.46]{HS},
$Lp$ is a minimal left ideal of $\beta S$, and so $\beta Sp=Lp$.

(3) Now suppose that $R$ contains a left cancelable element $t$ and let $p\in \bigcap_{x\in Z}U(x)$
Since $t$ is left cancelable in $\beta R$ by
\cite[Lemma 8.1]{HS} and $tp=tqp$ for some $q\in L$, it follows that $p=qp\in K(\beta S)$.
\end{proof}

Recall that a subset $A$ of a semigroup $S$ is {\it piecewise syndetic\/} if and only if
there is some $G\in\pf(S)$ such that for every $F\in\pf(S)$, there is some
$x\in S$ with $Fx\subseteq\bigcup_{t\in G}t^{-1}A$.  The important fact about
piecewise syndetic sets is that they are the subsets of $S$ whose closure meets
$K(\beta S)$, \cite[Theorem 4.40]{HS}.

\begin{definition} 
$\Omega=\Omega_Z=\{x\in Z:\overline{x^{-1}[\{1\}]\cap S}\cap K(\beta S)=\emp\}$.
\end{definition}

Thus $\Omega=\{x\in Z:x^{-1}[\{1\}]\cap S\hbox{ is not piecewise syndetic in }S\}$.
Note that, since $K(\beta S)$ is usually not topologically closed,
we have by Theorem \ref{bigcapUx} that not all sets of the form
$U(x)$ are closed.

\begin{definition} Let $x\in Z$.
$N(x)=\{p\in \beta S:(\forall t\in R)(T_p(x)(t)=0)\}$.
\end{definition}

\begin{lemma}\label{UequalsN} Let $x\in Z$. Then $N(x)$ is closed
and $N(x)\subseteq U(x)$. If $N(x)=U(x)$, then $x\in \Omega$.
If $S$ is a left ideal of $R$, then $N(x)=U(x)$ if and only if $x\in\Omega$.
\end{lemma}

\begin{proof}
To see that $N(x)$ is closed, let $p\in\beta S\setminus N(x)$, pick 
$t\in R$ such that $T_p(x)(t)=1\}$, and let $A=\{s\in S:T_s(x)(t)=1\}$. Then
$A\in p$ and $\overline A\cap N(x)=\emp$.

If $T_p(x)$ is constantly equal to $0$ on $R$, then $T_p(x)$ is uniformly
recurrent and thus $p\in U(x)$.

Let $A=x^{-1}[\{1\}]\cap S$.

First assume that $N(x)=U(x)$
and suppose that $x\notin\Omega$.  Since
$\overline A\cap K(\beta S)\neq \emptyset $, pick $p\in\overline A\cap K(\beta S)$. 
By Corollary \ref{U(x)ideal}(3), $p\in U(x)$ and
so for all $t\in R$, $T_p(x)(t)=0$.   Since $K(\beta S)$ is a union of groups, there exists
$q\in K(\beta S)$ such that $qp=p$. Pick $t\in S$ such that $t^{-1}A\in p$.
Also $T_p(x)(t)=0$ so $\{s\in S:x(ts)=0\}\in p$. Pick $s\in t^{-1}A$ such that
$x(ts)=0$, a contradiction.

Now assume that $S$ is a left ideal in $R$. Let $x\in\Omega$ and let $p\in U(x)$. We claim that $p\in N(x)$.
  To see this, suppose we have some $t\in R$ such that
$T_p(x)(t)=1$.  By Lemma \ref{charUx}, there exists an idempotent $q\in K(\beta S)$ such that
$\widetilde x(tqp)=1$. By \cite[Theorem 2.17]{HS}, $\beta S$ is a left ideal of $\beta R$ so
$tqp\in\beta S$ and so $A\in tqp=tqqp$.  Thus there is some $s\in S$ such that
$tsqp\in\overline A$.  Since $ts\in S$, $tsqp\in K(\beta S)$, a contradiction.
 \end{proof}

\begin{lemma}\label{bigcapclUx}  Let $p\in \bigcap_{x\in\Omega}U(x)$ and let $t\in R$.  If $tp\in \beta S$,
 then $tp\in \cl K(\beta S)$. In particular, $\beta Sp\subseteq \cl K(\beta S)$.
\end{lemma}

\begin{proof} 
 Assume that $tp\in \beta S\setminus \cl(K\beta S)$.
 We can choose $A\in tp$ such that $A\subseteq S$ and
$\overline A\cap K(\beta S)=\emptyset$. Let $x$ be the characteristic function of $A$
in $R$, so that $x\in \Omega$ and hence $p\in U(x)$. Observe that
$\wdd x$ is the characteristic funcion of $\cl_{\beta R}(A)$ in $\beta R$ and 
that $\cl_{\beta R}(A)\subseteq \beta S$, because $\beta S$ is
clopen in $\beta R$. Since $\wdd x(tp)=1$, it follows from Lemma \ref{charUx}(b) that
there exists $q\in K(\beta S)$ such that $\wdd x(tqp)=1$, and so $A\in tqp$. Now
$\{r\in \beta S:tqr\in \beta S\}$ is non-empty and is a right ideal of $\beta S$. 
There exists an idempotent $u$ in the intersection of this right ideal with the left ideal 
$\beta Sq$ of $\beta S$. Since $q\in \beta Su$, $qu=q$. So $tqp=tquup\in K(\beta S)$,
because $tqu\in\beta S$ and $u\in\beta Sq\subseteq K(\beta S)$. 
This contradicts the assumption that $\overline A\cap K(\beta S)=\emptyset$.
\end{proof}

\begin{corollary}\label{clKbS}   Each of the following statements 
implies that $\bigcap_{x\in \Omega}U(x)\subseteq \cl K(\beta S)$.
\begin{itemize}
\item[(a)]  There exists $e\in R$ such that $es=s$ for every $s\in S$.
\item[(b)] $S$ contains a left cancelable element.
\end{itemize}
\end{corollary}

\begin{proof} It follows from Lemma \ref{bigcapclUx} that 
(a) implies that  $\bigcap_{x\in \Omega}U(x)\subseteq \cl K(\beta S)$.
So assume that $s$ is a left cancelable element in $S$ and let
$p\in\bigcap_{x\in \Omega}U(x)$. By \cite[Lemma 8.1]{HS}, $s$ is left cancelable in
$\beta S$.
By Lemma \ref{bigcapclUx}, $sp\in \cl K(\beta S)$. Now 
$s\beta S=\overline{sS}$ is clopen in $\beta S$. 
So $sp\in \cl(K(\beta S)\cap s\beta S)$. 
We claim that, if $q\in K(\beta S)\cap s\beta S$, then $q\in sK(\beta S)$. 
To see this, suppose that $q\in K(\beta S)$ and that  $q=sv$ for some $v\in \beta S$. 
There is an idempotent $u\in K(\beta S)$ for which $qu=q$. This 
implies that $sv=svu$ and hence that $v=vu\in K(\beta S)$. 
So $sp\in \cl\big(sK(\beta S)\big)= s\cl K(\beta S)$ and hence $p\in \cl K (\beta S)$.
\end{proof}

\begin{corollary}\label{bigcapUandclK}  Assume that $S$ is a left ideal of $R$.
 Then each of the hypotheses (a) and (b) of Corollary {\rm \ref{clKbS}} implies that
$\bigcap_{x\in \Omega}U(x)=\cl K(\beta S)$.
\end{corollary}

\begin{proof} Assume that one of the hypotheses of Corollary \ref{clKbS} holds.
Then $$\textstyle\bigcap_{x\in \Omega}U(x)\subseteq \cl K(\beta S)\,.$$
To see that $\cl K(\beta S)\subseteq\bigcap_{x\in \Omega}U(x)$, let
$x\in\Omega$ be given. By Lemma \ref{UequalsN}, $U(x)=N(x)$ and so $U(x)$ is closed.
By Corollary \ref{U(x)ideal}(3), $K(\beta S)\subseteq U(x)$ and hence $\cl K(\beta S)\subseteq U(x)$. 
\end{proof}

For the statement of the following corollary we depart from
our standing assumptions about $R$, $S$, and $(Z,\langle T_s\rangle_{s\in S})$.

\begin{corollary} \label{bigcapequal} Let $S$ 
be a semigroup. There exist a dynamical system\break 
$(X,\langle T_s\rangle_{s\in S})$ and a subset
$M$ of $X$ such that $K(\beta S)=\bigcap_{x\in X}U(x)$
and $\cl K(\beta S)\break =\bigcap_{x\in M}U(x)$.
\end{corollary} 

\begin{proof} If $S$ has a left identity, let $R=S$.
Otherwise, let $R=S\cup\{e\}$ where $e$ is an identity
adjoined to $S$. The conclusion then follows from
Theorem \ref{bigcapUx} and Corollary\ref{bigcapUandclK}.
\end{proof}

In the proof of the above corollary, we could have simply
let $R=S\cup\{e\}$ where $e$ is an identity
adjoined to $S$, regardless of whether $S$ has a left identity,
as was done in \cite[Theorem 19.27]{HS} to produce a dynamical
system for any semigroup $S$ establishing the equivalence
of the notions of {\it central\/} and {\it dynamically central\/}.
We shall investigate the relationship between the systems
with phase space $X=\xtotwo{R}$ and $Y=\xtotwo{S}$ in the next
section.

We note that it is possible that 
$\bigcap_{x\in Z}U(x)\neq K(\beta S)$ and there is no subset $M$
of $Z$ such that $\bigcap_{x\in M}U(x)=\cl K(\beta S)$.  To see this,
let $S$ be an infinite zero semigroup.  That is, there is an element
$0\in S$ such that $st=0$ for all $s$ and $t$ in $S$. Then
$pq=0$ for all $p$ and $q$ in $\beta S$ and so 
$\cl K(\beta S)=K(\beta S)=\{0\}$.  Let $R=S$. Given
$x\in T$, if $a=x(0)$, then for all $p\in \beta S$,
$T_p(x)$ is constantly equal to $a$ and so $T_p(x)$ is uniformly recurrent.
That is, for any $x\in Z$, $U(x)=\beta S$.

In \cite{DH} it was shown that $\cl K(\beta\ben)$ is the intersection
of all of the closed two sided ideals that strictly contain it. In a 
similar vein, we would like to show that each $U(x)$ properly contains 
$K(\beta S)$.  One cannot hope for this to hold in general.  For 
example, as we have already noted, if $S$ is either left zero or right zero then so is
$\beta S$ and then $K(\beta S)=\beta S$.
Results establishing that $U(x)$ properly contains $K(\beta S)$
require some weak cancellation assumptions.

\begin{definition}\label{solutionsets} Let $S$ be a semigroup
and let $A\subseteq S$.
\begin{itemize}
\item[(a)] $A$ is a {\it left solution set\/} if and only if
there exist $u$ and $v$ in $S$ such that $A=\{x\in S:ux=v\}$.
\item[(b)] $A$ is a {\it right solution set\/} if and only if
there exist $u$ and $v$ in $S$ such that $A=\{x\in S:xu=v\}$.
\end{itemize}
\end{definition}

As is standard, we denote by $\omega$ the first infinite ordinal, which
is also the first infinite cardinal. That is, $\omega=\aleph_0$.

\begin{definition}\label{weakcanc} Let $S$ be a semigroup
with $|S|=\kappa\geq\omega$.
\begin{itemize}
\item[(a)] $S$ is {\it weakly left cancellative\/} if and only if
every left solution set in $S$ is finite.
\item[(b)] $S$ is {\it weakly right cancellative\/} if and only if
every right solution set in $S$ is finite.
\item[(c)] $S$ is {\it weakly cancellative\/} if and 
only if $S$ is both weakly left cancellative and weakly right cancellative. 
\item[(d)] $S$ is {\it very weakly left cancellative\/} if and only if 
the union of any set of fewer than $\kappa$ left solution sets has cardinality
less than $\kappa$.
\item[(e)] $S$ is {\it very weakly right cancellative\/} if and only if 
the union of any set of fewer than $\kappa$ right solution sets has cardinality
less than $\kappa$.
\item[(f)] $S$ is {\it very weakly cancellative\/} if and 
only if $S$ is both very weakly left cancellative and very weakly right cancellative. 
\end{itemize}
\end{definition}

Given a set $X$ and a cardinal $\kappa$, we let
$U_\kappa(X)$ be the set of $\kappa$-uniform ultrafilters on $X$.
That is, $U_{\kappa}(X)=\{p\in\beta X:(\forall A\in p)(|A|\geq\kappa)\}$.

\begin{theorem}\label{proper} Assume that  $|R|=|S|=\kappa\geq\omega$,  
$S$ is very weakly cancellative, and has the property that
$|\{e\in S:(\exists s\in S)(es=s)\}|<\kappa$.  Then for all $x\in Z$,
$U(x)\cap U_\kappa(S)\setminus \cl K(\beta S)\neq\emp$.
\end{theorem}

\begin{proof} Let $E=\{e\in S:(\exists s\in S)(es=s)\}$.
 Let $x\in Z$ and pick $q\in K(\beta S)$.  Let $y=T_q(x)$.  By
Corollary \ref{U(x)ideal}(3), $y$ is uniformly recurrent.  For
each $F\in\pf(R)$, let $B_F=\big\{s\in S:(\forall t\in F)\big(x(ts)=y(t)\big)\big\}$.
Since $$\textstyle B_F=\{s\in S:T_s(x)\in\bigcap_{t\in F}\pi_t^{-1}[\{y(t)\}]\}\,,$$ we have
$B_F\in q$.  By \cite[Lemma 6.34.3]{HS},  $K(\beta S)\subseteq U_\kappa(S)$ and so $|B_F|=\kappa$.
Note that if $F\subseteq H$, then $B_H\subseteq B_F$.

Enumerate $\pf(R)$ as $\langle F_\alpha\rangle_{\alpha<\kappa}$.
Choose $t_0\in B_{F_0}\setminus E$.  Let $0<\alpha<\kappa$ and assume that
we have chosen $\langle t_\delta\rangle_{\delta<\alpha}$ satisfying
the following inductive hypotheses.

\begin{itemize}
\item[(1)] For each $\delta<\alpha$, $t_\delta\in B_{F_\delta}$.
\item[(2)] For each $\delta<\alpha$, $FP(\langle t_\beta\rangle_{\beta\leq\delta})\cap E=\emp$.
\item[(3)] For each $\delta<\alpha$, if $\delta>0$, then 
$t_\delta\notin FP(\langle t_\beta\rangle_{\beta<\delta})$.
\item[(4)] For each $\delta<\alpha$, if $\delta>0$,
$s\in FP(\langle t_\beta\rangle_{\beta<\delta})$, and
$\gamma<\delta$, then $st_\delta\neq t_\gamma$.
\end{itemize}

The hypotheses are satisfied for $\delta=0$.  Let 
\begin{eqnarray*}
M_0&=&\textstyle \{t\in S:\big(\exists H\in\pf(\alpha)\big)\big((\prod_{\beta\in H}t_\beta)t\in E\big)\big\}\}
\hbox{\rm\ and let}\\
M_1&=& \{t\in S:\big(\exists s\in FP(\langle t_\beta\rangle_{\beta<\alpha})\big)
(\exists \gamma<\alpha)(st=t_\gamma)\}\,.
\end{eqnarray*}
  
Note that
$|FP(\langle t_\beta\rangle_{\beta<\alpha})|\leq |\pf(\alpha)|<\kappa$.
Also, given $H\in\pf(\alpha)$ and $s\in E$, $\{t\in S:(\prod_{\beta\in H}t_\beta)t=s\}$
is a left solution set so $|M_0|<\kappa$.  Note also that, 
given $s\in FP(\langle t_\beta\rangle_{\beta<\alpha})$ and $\gamma<\alpha$, 
$\{t\in S:st=t_\gamma\}$ is a left solution set so $|M_1|<\kappa$.  Thus
we may choose $t_\alpha\in B_{F_\alpha}\setminus(E\cup FP(\langle t_\beta\rangle_{\beta<\alpha})
\cup M_0\cup M_1)$.  The induction hypotheses are satisfied for $\alpha$.

Let $B=\{t_\alpha:\alpha<\kappa\}$ and let $C=\bigcap_{\alpha<\kappa}\cl FP(\langle t_\beta\rangle_{\alpha<\beta<\kappa})$.
By \cite[Theorem 4.20]{HS}, $C$ is a compact subsemigroup of $\beta S$. We claim that
$\overline B\cap K(C)=\emp$.  Suppose instead that we have
$p\in \overline B\cap K(C)$.  Pick $r\in K(C)$ such that $p=pr$. (By \cite[Lemma 1.30]{HS}, an
idempotent in the minimal left ideal $L$ of $C$ in which $p$ lies will do.)
Let $D=\{s\in S:s^{-1}B\in r\}$.  Then $D\in p$ so $D\cap B\neq\emp$ so
pick $\alpha<\kappa$ such that $t_\alpha^{-1} B\in r$.
Then $\overline{t_\alpha^{-1} B}\cap FP(\langle t_\beta\rangle_{\alpha<\beta<\kappa})\neq
\emp$ so pick finite $H\subseteq\{\beta:\alpha<\beta<\kappa\}$ such that
$\prod_{\beta\in H}t_\beta\in t_{\alpha}^{-1}B$.  Pick
$\gamma<\kappa$ such that $t_\alpha\prod_{\beta\in H}t_\beta=t_\gamma$. Let
$\max H=\mu$ and let $K=H\setminus\{\mu\}$. 
If $K=\emp$, then $t_\alpha t_\mu=t_\gamma$.  If $K\neq\emp$, then
$t_\alpha(\prod_{\beta\in K}t_\beta) t_\mu=t_\gamma$.
 If $\gamma>\mu$ we get
a contradiction to hypothesis (3).  If $\mu=\gamma$ we either
get $t_\alpha\in E$ or $t_\alpha\prod_{\beta\in K}t_\beta\in E$, contradicting
hypothesis (2). If $\gamma<\mu$ we get a contradiction to hypothesis (4).
Thus $\overline B\cap K(C)=\emp$ as claimed.

Now we claim that $\overline B\cap K(\beta S)=\emp$.  Suppose 
instead we have $p\in\overline B\cap K(\beta S)$. By \cite[Lemma 6.34.3]{HS}
we have that $p\in U_\kappa(S)$ and consequently, $p\in C$.
Thus $K(\beta S)\cap C\neq\emp$ and so, by \cite[Theorem 1.65]{HS}, 
$K(C)=K(\beta S)\cap C$, contradicting the fact that $\overline B\cap K(C)=\emp$.
Since $\overline B$ is clopen, we thus have $\overline B\cap \cl K(\beta S)=\emp$.

Now let ${\cal C}=\{B_F:F\in\pf(S)\}\cup\{B\}$.  We claim that 
${\cal C}$ has the $\kappa$-uniform finite intersection property.
To see this, let ${\cal F}\in\pf\big(\pf(S)\big)$ and let
$H=\bigcup{\cal F}$. If $\delta<\kappa$ and
$H\subseteq F_\delta$, then $t_\delta\in B\cap\bigcap_{F\in {\cal F}}B_F$.
Since $|\{\delta<\kappa:H\subseteq F_\delta\}|=|\{F\in\pf(S):H\subseteq F\}|=\kappa$,
we have that $|\bigcap{\cal C}|=\kappa$ as required.  Pick by \cite[Corollary 3.14]{HS}
$p\in U_\kappa(S)$ such that ${\cal C}\subseteq p$.  

Since $B_F\in p$ for each $F\in\pf(R)$, we have $T_p(x)=y$ so $p\in U(x)$. 
Since $B\in p$, $p\notin \cl K(\beta S)$.\end{proof}

\begin{corollary}\label{rightcanc}
Assume that $|R|=|S|=\kappa\geq\omega$ and that 
$S$ is right cancellative and very weakly left cancellative. 
Then for all $x\in Z$,
$U(x)\cap U_\kappa(S)\setminus \cl K(\beta S)\neq\emp$.
\end{corollary}

\begin{proof} Let $E=\{e\in S:(\exists s\in S)(es=s)\}$.
It suffices to show that $|E|<\kappa$.  Pick
$x\in S$. Given $e\in E$ and $s\in S$ such that $es=s$,
we have that $xes=xs$ so $xe=x$. Thus
$E$ is contained in the left solution set $\{t\in S:xt=x\}$.
\end{proof}

\begin{corollary}\label{Kproper}
Assume that $|R|=|S|=\kappa\geq\omega$, that 
$S$ is very weakly cancellative, that $S$ has a member $e$ such that $es=s$ for all $s\in S$, 
and $|\{e\in S:(\exists s\in S)(es=s)\}|<\kappa$. Then $K(\beta S)=\bigcap_{x\in Z}U(x)$ and
for each $x\in Z$, $U(x)$ properly contains $K(\beta S)$.
\end{corollary}

\begin{proof} This is an immediate consequence of Theorems \ref{bigcapUx} and 
\ref{proper}. \end{proof}

\begin{corollary}\label{clKproper}
Assume that $S$ is a left ideal 
of $R$, $|R|=|S|=\kappa\geq\omega$, $S$ is very weakly cancellative, 
$S$ has a member $e$ such that $es=s$ for all $s\in S$, 
and $|\{e\in S:(\exists s\in S)(es=s)\}|<\kappa$. Then $\cl K(\beta S)=\bigcap_{x\in \Omega}U(x)$ and
for each $x\in \Omega$, $U(x)$ properly contains $\cl K(\beta S)$.
\end{corollary}

\begin{proof} By Corollary \ref{bigcapUandclK} $\cl K(\beta S)=\bigcap_{x\in \Omega}U(x)$.
By Theorem \ref{proper}, for each $x\in \Omega$, $U(x)$ properly contains $\cl K(\beta S)$.
\end{proof}

\section{Relations between systems with phase spaces $X$ and $Y$}

Throughout this section we will let $S$ be an arbitrary semigroup and
let $Q=S\cup\{e\}$, where $e$ is an identity adjoined to $S$, even
if $S$ already has an identity.  We will let $(X,\langle T_{X,s}\rangle_{s\in S})$
be the dynamical system of Lemma {\rm \ref{stotwodyn}} determined by $R=Q$
let $(Y,\langle T_{Y,s}\rangle_{s\in S})$
be the dynamical system of Lemma {\rm \ref{stotwodyn}} determined by $R=S$.
For $x\in X$ we will let $U_X(x)=\{p\in\beta S:T_{X,p}(x)$ is uniformly recurrent$\}$
and let $U_Y(x)=\{p\in\beta S:T_{Y,p}(x)$ is uniformly recurrent$\}$.

We have from the results of the previous section that for any semigroup
$S$, $K(\beta S)=\bigcap_{x\in X}U_X(x)$ and $\cl K(\beta S)=\bigcap_{x\in \Omega_X}U_X(x)$.
We are interested in determining when the corresponding assertions hold with
respect to $Y$.  Of course, the simplest situation in which they do is
when for each $x\in X$, $U_X(x)=U_Y(x_{|S})$ so we address this problem first,
beginning with the following simple observation.

\begin{lemma}\label{uxinuy} Let $x\in X$. Then $U_X(x)\subseteq U_Y(x_{|S})$.
\end{lemma}

\begin{proof} Let $y=x_{|S}$ and note that 
$\widetilde y$ is the restriction of $\widetilde x$ to $\beta S$.
Let $L$ be a minimal left ideal of $\beta S$.
By Lemma \ref{charUx}, 
$p\in U_X(x)$ if and only if there exists  $q\in L$, such 
that $\widetilde x(tp)=\widetilde x(tqp)$ for all $t\in Q$. And $p\in U_Y(x_{|S})$ 
if and only if there exists $q\in L$ such that $\widetilde y(tp)=\widetilde y(tqp)$ for
all $t\in S$. \end{proof}

\begin{theorem}\label{uxisuy} The following statements are equivalent.
\begin{itemize}
\item[(a)] For all $x\in X$, $U_X(x)=U_Y(x_{|S})$.
\item[(b)] There do not exist $p\in\beta S$ and $x\in X$ such
that $T_{X,p}(x)$ is the characteristic function of $\{e\}$ in $X$.
\item[(c)] For every $p\in\beta S$, $p\in\beta Sp$.
\end{itemize}
\end{theorem}

\begin{proof} Assume that (a) holds and suppose we have 
$p\in\beta S$ and $x\in X$ such
that $T_{X,p}(x)$ is the characteristic function of $\{e\}$ in $X$. Then $T_{Y,p}(x_{|S})$ is constantly 
$0$ so $p\in U_Y(x_{|S})$.  But
$V=\{u\in X:w(e)=1\}$ 
is a neighborhood of $w=T_{X,p}(x)$ in $X$, while
$\{s\in S:T_{X,s}(w)\in V\}=\emp$, so
$p\notin U_X(x)$.

To see that (b) implies (c), assume that (b) holds and suppose that
we have some $p\in\beta S$ such that $p\notin\beta Sp$.  Since
$\beta Sp=\rho_p[\beta S]$, $\beta Sp$ is closed. Pick
$A\in p$ such that $\overline A\cap \beta Sp=\emp$. Let $x$ be the
characteristic function of $A$ in $X$.
First let $s\in S$. 
Then $sp\notin \overline A$ so $s^{-1}(S\setminus A)\in p$ so to
see that $T_{X,p}(s)=0$, it
suffices to observe that $s^{-1}(S\setminus A)\subseteq \{t\in S:T_{X,t}(x)(s)=0\}$.
Since $A\in p$ and for $t\in A$, $T_{X,t}(x)(e)=x(t)=1$, we have that
$T_{X,p}(x)(e)=1$. 

By Lemma \ref{uxinuy},  we have $U_X(x)\subseteq U_Y(x_{|S})$ for all $x\in X$, so
to show that (c) implies (a), it suffices to let $x\in X$, let $p\in U_Y(x_{|S})$, assume
that $p\in \beta Sp$,  and show
that $p\in U_X(x)$.  By Lemma \ref{charUxp}, it suffices to let $L$ be a
minimal left ideal of $\beta S$ and let $F\in\pf(Q)$ and show that there
is some $q\in L$ such that $\widetilde x(tp)=\widetilde x(tqp)$ for every
$t\in F$. For $t\in F$, let $B_t=\{s\in S:
x(ts)=\widetilde x(tp)\}$.  Then $\bigcap_{t\in F}B_t\in p$ 
and $p\in\beta Sp=\cl(Sp)$ so pick $v\in S$ such that
$\bigcap_{t\in F}B_t\in vp$. Let $y=x_{|S}$. Now $Fv\in\pf(S)$ so pick by Lemma \ref{charUxp}
$q\in L$ such that for all $t\in F$,
$\widetilde y(tvp)=\widetilde y(tvqp)$.
Let $q'=vq$ and note that $q'\in L$.  Let $t\in F$ be given.
Then $B_t\in vp$ so $\widetilde x(tvp)=\widetilde x(tp)$ and thus
$\widetilde x(tp)=\widetilde y(tvp)=\widetilde y(tvqp)=
\widetilde x(tq'p)$. \end{proof}

\begin{corollary}\label{corKisU} If for all $p\in\beta S$, $p\in\beta Sp$,
then $K(\beta S)=\bigcap_{x\in Y}U_Y(x)$ and
$\cl K(\beta S)=\bigcap_{x\in \Omega_Y}U_Y(x)$.
\end{corollary}

\begin{proof} The first assertion is an immediate consequence of Theorems \ref{bigcapUx} and \ref{uxisuy}.
The second assertion follows from Corollary \ref{bigcapUandclK} and Theorem \ref{uxisuy}.
\end{proof}

We have already mentioned the problem of determining whether
$K(\beta S)$ or $\cl K(\beta S)$ is prime.  Recall that
an ideal $I$ in a semigroup is {\it semiprime\/} if and only if
whenever $ss\in I$, one must have $s\in I$.

\begin{corollary}\label{primeK} \begin{itemize}
\item[(1)] If $K(\beta S)\neq \bigcap_{x\in Y}U_Y(x)$, then 
$K(\beta S)$ is not semiprime.
\item[(2)] If $\cl K(\beta S)\neq \bigcap_{x\in \Omega_Y}U_Y(x)$, then 
$\cl K(\beta S)$ is not semiprime.
\end{itemize}
\end{corollary}

\begin{proof}
(1) If $p\in \bigcap_{x\in Y}U_Y(x)\setminus K(\beta S)$, then $pp\in\beta Sp$ and by 
Theorem \ref{bigcapUx}, $\beta Sp\subseteq K(\beta S)$.

(2) If $p\in \bigcap_{x\in \Omega_Y}U_Y(x)\setminus \cl K(\beta S)$, then $pp\in\beta Sp$
and by Lemma \ref{bigcapclUx}, $\beta Sp\subseteq \cl K(\beta S)$.
\end{proof}

By virtue of Theorem \ref{uxisuy} we are interested in knowing when there
is some $p\in\beta S$ such that $p\notin \beta Sp$.

\begin{lemma}\label{lempnotin} Let $p\in\beta S$. Then $p\notin \beta Sp$ if and only if
there exists $A\subseteq S$ such that for all $x\in S$, $x^{-1}A\in p$
and $A\notin p$.
\end{lemma}

\begin{proof} Let $C(p)=\{A\subseteq S:(\forall x\in S)(x^{-1}A\in p)\}$.
By \cite[Theorem 6.18]{HS}, $p\in \beta Sp$ if and only if $C(p)\subseteq p$.
\end{proof}

\begin{theorem}\label{thmpnotin} Assume that 
$|S|=\kappa\geq\omega$.  There exists $p\in\beta S$ such that
$p\notin \beta Sp$ if and only if there exists 
$\langle y_F\rangle_{F\in\pf(S)}$ in $S$ such that\hfill\break
$\{y_F:F\in\pf(S)\}\cap\bigcup\{Fy_F:F\in\pf(S)\}=\emp$.
\end{theorem}

\begin{proof} Necessity. Pick $p\in\beta S$ such that
$p\notin \beta Sp$. By Lemma \ref{lempnotin}, pick $A\subseteq S$ such that
for all $x\in S$, $x^{-1}A\in p$ and $A\notin p$.  For
$F\in\pf(S)$ pick $y_F\in (S\setminus A)\cap\bigcap_{x\in F}x^{-1}A$.

Sufficiency. Let $A=\bigcup\{Fy_F:F\in\pf(S)\}$.  Then
$\{S\setminus A\}\cup\{x^{-1}A:x\in S\}$ has the finite intersection
property so pick $p\in\beta S$ such that $\{S\setminus A\}\cup\{x^{-1}A:x\in S\}\subseteq
p$.  By Lemma \ref{lempnotin}, $p\notin \beta Sp$. \end{proof}

One of the assumptions in the following corollary is that $S^*=\beta S\setminus S$ is a right ideal
of $\beta S$. A (not very simple) characterization of when $S^*$ is a right ideal
of $\beta S$ is given in \cite[Theorem 4.32]{HS}. By \cite[Corollary 4.33 and Theorem 4.36]{HS}  it is
sufficient that $S$ be either right cancellative or weakly cancellative.

\begin{corollary}\label{corpnotin} Assume that  $|S|=\kappa\geq\omega$
and assume that $$|S\setminus\{t\in S:(\exists s\in S)(st=t)\}|=\kappa\,.$$  If
either $S^*$ is a right ideal of $\beta S$ or $S$ is very weakly left cancellative, 
then there exists $p$ in $\beta S$ such that $p\notin \beta Sp$.
\end{corollary}

\begin{proof} Assume first that $S^*$ is a right ideal of $\beta S$, and pick 
$t\in S$ such that there is no $s\in S$ with $st=t$.  Then $t\notin St$ and 
$t\notin S^* t$.

Now assume that $S$ is very weakly left cancellative.
 Enumerate $\pf(S)$ as $\langle F_\alpha\rangle_{\alpha<\kappa}$.  By
Theorem \ref{thmpnotin}, it suffices to produce $\langle t_\alpha\rangle_{\alpha<\kappa}$ in
$S$ such that\break $\{t_\alpha:\alpha<\kappa\}\cap\bigcup\{F_\alpha t_\alpha:\alpha<\kappa\}=\emp$.

Let $E=\{t\in S:(\exists s\in S)(st=t)\}$. Pick $t_0\in S\setminus E$.  Let $0<\alpha<\kappa$ and assume
we have chosen $\langle t_\delta\rangle_{\delta<\alpha}$ in $S\setminus E$ such that
if $\delta>0$, then $t_\delta\notin\bigcup_{\mu<\delta}F_{\mu} t_\mu$ and for each 
$x\in F_\delta$, $xt_\delta\notin\{t_\mu:\mu<\delta\}$.

For $x\in S$ and $\mu<\alpha$, let $H_{x,\mu}=\{t\in S:xt=t_\mu\}$.  Then
each $H_{x,\mu}$ is a left solution set so $|\bigcup\{H_{x,\mu}:x\in F_\alpha\hbox{ and }
\mu<\alpha\}|<\kappa$.  Pick 
$$\textstyle t_\alpha\in S\setminus (E\cup\bigcup\{H_{x,\mu}:x\in F_\alpha\hbox{ and }
\mu<\alpha\}\cup\bigcup_{\mu<\alpha}F_\mu t_\mu)\,.$$

Suppose we have some $\mu<\kappa$ such that
$t_\mu\in \bigcup\{F_\alpha t_\alpha:\alpha<\kappa\}$ and pick
$\alpha<\kappa$ and $x\in F_\alpha$ such that
$t_\mu=xt_\alpha$.  Then $\alpha\neq\mu$ because $t_\alpha\notin E$.
If $\alpha<\mu$, we would have $t_\mu\in F_\alpha t_\alpha$.
So we must have $\mu<\alpha$.  But then $t_\alpha\in H_{x,\mu}$, a contradiction.
\end{proof}

We conclude this section by exhibiting a sufficient condition which guarantees
$K(\beta S)=\bigcap_{x\in Y}U_Y(x).$  We shall see that this
does not require equality between $U_X(x)$ and $U_Y(x_{|S})$ for
all $x\in X$.

\begin{theorem}\label{sAimplies} Assume that for all
$p\in\bigcap_{x\in Y}U_Y(x)$ and all $A\in p$  
the assumption that  $\{t\in S:t^{-1}sA\in p\}$ is syndetic for every $s\in S$,
implies that $\{t\in S:t^{-1}A\in p\}\neq \emptyset$. Then
$K(\beta S)=\bigcap_{x\in Y}U_Y(x)$.\end{theorem}

\begin{proof} Assume that $p\in \bigcap_{x\in Y}U_Y(x)\setminus K(\beta S)$. By Theorem \ref{bigcapUx}(2),
$\beta Sp\subseteq K(\beta S)$ so $p\notin \beta Sp$.  Pick $A\in p$ such that $\overline A\cap \beta Sp=\emptyset$. 
Thus $\{t\in S:t^{-1}A\in p\}=\emptyset$.
We claim that for all $s\in S$, $\{t\in S:t^{-1}sA\}$ is syndetic. So let $s\in S$.
By \cite[Theorem 4.48]{HS} it suffices to let $L$ be a minimal left ideal of $\beta S$
and show that there is some $q\in L$ such that $\{t\in S:t^{-1}sA\in p\}\in q$.
By Theorem \ref{bigcapUx}(2), $sp\in Lp$ so pick $q\in L$ such that $sp=qp$.
Then $sA\in qp$ so $\{t\in S:t^{-1}sA\in p\}\in q$ as required.
\end{proof}

Note that by Theorem \ref{bigcapUx}(3), 
$K(\beta \ben,+)=\bigcap_{x\in Y}U_Y(x)$ while $1\notin \beta \ben+1$ so by 
Theorem \ref{uxisuy}, it is not the case that for all $x\in X$, $U_X(x)=U_Y(x_{|S})$.
On the other hand, given $p\in K(\beta\ben,+)$ one has $p=q+p$ for some
$p\in K(\beta\ben,+)$ so automatically for any $A\in p$,
$\{t\in \ben:-t+A\in p\}\neq \emptyset$ so the hypotheses of Theorem \ref{sAimplies} are valid.

\section{Recurrence and surjectivity of $T_p$}

So far in this paper we have been considering the notion
of uniform recurrence. We now introduce a notion which is usually weaker.

\begin{definition}\label{defrecur} Let $(X,\langle T_s\rangle_{s\in S})$ be a
dynamical system. The point $x\in X$ is {\it recurrent\/} if and only if
for each neighborhood $V$ of $x$ in $X$, $\{s\in S:T_s(x)\in V\}$ is infinite.
\end{definition}

If all syndetic subsets of a semigroup $S$ are infinite, then any
uniformly recurrent point of $X$ is recurrent.  This is not always the
case.  For example, if $S$ is a left zero semigroup and $x\in S$,
then $x$ is uniformly recurrent in the dynamical system $(\beta S, \langle \lambda_s\rangle_{s\in S})$
but is not recurrent.  (We have that $\{x\}$ is a neighborhood of $x$ and
$\big\{s\in S:\lambda_s(x)\in\{x\}\big\}=\{x\}$, which is syndetic, but finite.)

The following characterization of recurrence is very similar to the
characterization of uniform recurrence in \cite[Theorem 19.23]{HS}.
Part of the results depend on the assumption that $S^*$ is a subsemigroup of $\beta S$.
There is a characterization of $S^*$ as a subsemigroup in \cite[Theorem 4.28]{HS}.
By \cite[Corollary 4.29 and Theorem 4.31]{HS} it is sufficient that
$S$ be right cancellative or weakly left cancellative.

\begin{theorem}\label{charrec} Let $(X,\langle T_s\rangle_{s\in S})$ be a
dynamical system.  Statements (a) and (b) are equivalent and imply
statements (c) and (d), which are equivalent. If $S^*$ is a subsemigroup
of $\beta S$, then all four statements are equivalent.
\begin{itemize}
\item[(a)] There exists an idempotent $p\in S^*$ such that
$T_p(x)=x$.
\item[(b)] There exist $y\in X$ and an idempotent $p\in S^*$ such that
$T_p(y)=x$.
\item[(c)] There exists $p\in S^*$ such that
$T_p(x)=x$.
\item[(d)] $x$ is recurrent.
\end{itemize}\end{theorem}

\begin{proof} Trivially (a) implies (b) and (a) implies (c). To see that
(b) implies (a), pick $y\in X$ and an idempotent $p\in S^*$ such that
$T_p(y)=x$. Then $x=T_p(y)=T_{pp}(y)=T_p\big(T_p(y)\big)=T_p(x)$.

To see that (c) implies (d), pick $p\in S^*$ such that
$T_p(x)=x$. Let $V$ be a neighborhood of $x$. Then
$\{s\in S:T_s(x)\in V\}\in p$ so $\{s\in S:T_s(x)\in V\}$ is infinite.

To see that (d) implies (c), assume that $x$ is recurrent and for each neighborhood $V$ of $x$,
let $D_V=\{s\in S:T_s(x)\in V\}$.  Then any finite subfamily of
$\{D_V:V$ is a neighborhood of $x\}$ has infinite intersection so
pick by \cite[Corollary 3.14]{HS} some $p\in S^*$ such that
$\{D_V:V$ is a neighborhood of $x\}\subseteq p$. Then $T_p(x)=x$.

Now assume that $S^*$ is a semigroup. To see that (c) implies (a), pick $p\in S^*$ such that
$T_p(x)=x$ and let $E=\{q\in S^*:T_q(x)=x\}$.
Since $S^*$ is a subsemigroup of $\beta S$, we have that
$E$ is a subsemigroup of $\beta S$.  We claim that
$E$ is closed. To see this, let $q\in \beta S\setminus E$.
If $q\in S$, then $\{q\}$ is a neighborhood of $q$ missing $E$, 
so assume that $q\in S^*$. Pick an open neighborhood $U$ of $T_q(x)$
such that $x\notin \cl U$ and let $A=\{s\in S:T_s(x)\in U\}$.
Then $\overline A$ is a neighborhood of $q$ which misses $E$.
Since $E$ is a compact right topological semigroup, there 
is an idempotent in $E$.
\end{proof}

Recall that in any dynamical system, $(X,\langle T_s\rangle_{s\in S})$,
$K(\beta S)\subseteq\bigcap_{x\in X}U_X(x)$ and we have obtained
sufficient conditions for equality.

\begin{theorem}\label{qpinK} Let $(X,\langle T_s\rangle_{s\in S})$ be a dynamical system,
let $p\in \beta S$, and assume that $T_p:X\to X$ is surjective and
$K(\beta S)=\bigcap_{x\in X}U_X(x)$. Then for any $q\in\beta S$,
$qp\in K(\beta S)$ if and only if $q\in K(\beta S)$.
\end{theorem}

\begin{proof} Let $q\in\beta S$. The sufficiency is trivial, so
assume that $qp\in K(\beta S)$. It suffices to show 
that $q\in \bigcap_{x\in X}U(x)$, so let $x\in X$ be given.
Pick $y\in X$ such that $T_p(y)=x$. 
Then $T_q(x)=T_q\big(T_p(y)\big)=T_{qp}(y)$. Since
$qp\in U(y)$ we have $T_{qp}(y)$ is uniformly recurrent, and
so $T_q(x)\in U(x)$ as required.\end{proof}

\begin{definition}\label{defNS} Let $(X,\langle T_s\rangle_{s\in S})$ be a dynamical system.
Then $NS=NS_X=\{p\in\beta S:T_p$ is not surjective$\}$.
\end{definition} 

We have seen that $U(x)$ is always a left ideal of $\beta S$.

\begin{lemma}\label{NSright} Let $(X,\langle T_s\rangle_{s\in S})$ be a dynamical system.
If $NS\neq\emp$, then $NS$ is a right ideal of $\beta S$.
\end{lemma}

\begin{proof} Given $p\in NS$ and $q\in \beta S$, the 
range of $T_{pq}$ is contained in the range of $T_p$. \end{proof}

\begin{lemma}\label{idemNS}  Let $(X,\langle T_s\rangle_{s\in S})$ be a dynamical system.
If there is some $x\in X$ such that $x$ is not recurrent, then
$\{p\in S^*:pp=p\}\subseteq NS$.
\end{lemma}

\begin{proof} Pick $x\in X$ such that $x$ is not recurrent and
let $p$ be an idempotent in $S^*$.  We claim that
$x$ is not in the range of $T_p$, so suppose instead we have
$y\in X$ such that $T_p(y)=x$. Then by 
Theorem \ref{charrec}, $x$ is recurrent.\end{proof}

We shall establish a strong connection between the surjectivity of $T_p$ 
and $p$ being right cancelable in $\beta S$.  The 
purely algebraic result in Theorem \ref{rpcancel} will be useful.

\begin{lemma}\label{satb} Let $S$ be a countable right cancellative and 
weakly left cancellative semigroup and let $B$ be an infinite subset of $S$.
There is an infinite subset $D$ of $B$ with the property that whenever $s$ and $t$ are distinct 
members of $S$, there is a finite  subset $F$ of $D$ such that $sa\neq tb$ whenever $a,b\in D\setminus F$.
\end{lemma}

\begin{proof} Let $\Delta=\{(s,s):s\in S\}$ and enumerate $(S\times S)\setminus\Delta$ as
$\langle (s_n,t_n)\rangle_{n=1}^\infty$.  Pick $a_1\in B$. Assume $n\in \ben$ and we 
have chosen $\langle a_i\rangle_{i=1}^n$.  Let $W_n=\{b\in S:$ there exist $i,j\in\nhat{n}$
such that $s_ia_j=t_ib$ or $s_ib=t_ia_j\}$.  Then $W_n$ is the union of finitely many left solution sets, so
is finite.  Pick $a_{n+1}\in B\setminus(W_n\cup\{a_1,a_2,\ldots,a_n\})$.

Let $D=\{a_n:n\in\ben\}$. Let $s$ and $t$ be distinct members of $S$ and pick $n$ such that
$(s,t)=(s_n,t_n)$.  Let $F=\big\{a_i:i\in\nhat{n}\big\}$. To see that $F$ is as required,
let $a,b\in D\setminus F$ and suppose $sa=tb$. Then by right cancellation, $a\neq b$.
Pick $m>n$ and $r>n$ such that $a=a_m$ and $b=a_r$. If $m<r$, then $a_r\in W_{r-1}$.
If $r<m$, then $a_m\in W_{m-1}$.
\end{proof}

\begin{theorem}\label{rpcancel} Let $S$ be a countable cancellative semigroup. If $p\in \beta S\setminus K(\beta S)$, 
then there exists an infinite $D\subseteq S$ such that for every $r\in D^*$, $rp$ is right cancelable
in $\beta S$.\end{theorem}

\begin{proof} Choose $q\in K(\beta S)$. We first claim that for each $s\in S$, 
$sp\notin K(\beta S)$ and in particular, $sp\notin\beta Sqp$.  So suppose 
we have $sp\in K(\beta S)$. Then $sp$ is in a minimal left ideal $L$ of 
$\beta S$. Pick an idempotent $r\in L$. By \cite[Lemma 1.30]{HS}, $sp=spr$.
By \cite[Lemma 8.1]{HS} $s$ is left cancelable in $\beta S$ so $p=pr$, and thus
$p\in K(\beta S)$. This contradiction establishes the claim.  For each
$s\in S$, pick $U_s\in sp$ such that $\overline{U_s}\cap \beta Sqp=\emptyset$. 
For each $s,t\in S$, there exists $V_{s,t}\in q$
such that $\overline {U_s}\cap t\overline {V_{s,t}}p=\emptyset$ 
because $\lambda_t\circ\rho_p(q)\in \beta S\setminus\overline {U_s}$. 

By \cite[Theorem 3.36]{HS}, there exists an infinite subset $B$
of $S$ such that $B^*\subseteq \bigcap_{s,t\in S}\,\overline{V_{s,t}}$. 
Then for every $r\in B^*$ and every $s,t\in S$,
$trp\notin \overline {U_s}$.

By Lemma \ref{satb} pick an infinite subset $D$ of $B$ such that, whenever $s$ and $t$ are distinct elements
of $S$, there is a finite subset $F$ of $D$ such that $sa\neq tb$ whenever $a,b\in D\setminus F$.
Enumerate $D$ as $\langle d_n\rangle_{n=1}^\infty$ and for each distinct $s$ and $t$ in $S$,  pick
$n_{s,t}\in\ben$ such that $sd_m\neq td_n$ whenever $m,n>n_{s,t}$.

We claim that, for every $r\in D^*$, $rp$ is right cancelable in $\beta S$. 
We shall apply \cite[Theorem 3.40]{HS} three times.

Assume that $q_1rp=q_2rp$, where $q_1$ and $q_2$ are distinct elements of $\beta S$. Let $A_1$ and $A_2$
be disjoint subsets of $S$ which are members of $q_1$ and $q_2$ respectively. Since $q_1rp\in cl(A_1rp)$ and
$q_2rp\in cl(A_2rp)$, an application of \cite[Theorem 3.40]{HS} shows that 
either $A_1rp\cap cl(A_2rp)\neq \emp$ or $A_2rp\cap cl(A_1rp)\neq \emp$, and without loss of generality,
we may assume that the former holds.  Thus we have some $s\in A_1$ and $q'\in \overline{A_2}$ such that
$srp=q'rp$.  Now $srp\in\cl(sDp)$ and $q'rp\in\cl\big((S\setminus\{s\})rp\big)$, so either
$sDp\cap\cl\big((S\setminus\{s\})rp\big)\neq \emp$ or $(S\setminus\{s\})rp\cap\cl(sDp)\neq\emp$.
We thus have either
\begin{itemize}
\item[(i)] $sDp\cap \cl\big((S\setminus\{s\})rp\big)\neq \emp$, in which case
we choose $d\in D$ and $y\in\beta S$ such that $sdp=yrp$; or
\item[(ii)] $sDp\cap \cl\big((S\setminus\{s\})rp\big)= \emp$, in which case we pick
$t\in S\setminus\{s\}$ and $r'\in\overline{D}$ such that $sr'p=trp$. Since $sDp\cap \cl\big((S\setminus\{s\})rp\big)= \emp$,
we have $r'\in D^*$.
\end{itemize}

Suppose that (i) holds.  Then $U_{sd}\in sdp$ so $\{v\in S:v^{-1}U_{sd}\in rp\}\in y$, so
pick $v\in S$ such that $U_{sd}\in vrp$.  But $r\in V_{sd,v}$, so this is a contradiction.
Thus (ii) holds.

Now $sr'p\in\cl\{sd_mp:m>n_{s,t}\}$ and 
$trp\in\cl\{td_mp:m>n_{s,t}\}$ so, essentially without loss of generality,
we have $\{sd_mp:m>n_{s,t}\}\cap\cl\{td_mp:m>n_{s,t}\}\neq\emp$.  (We have distinguished
between $s$ and $t$ at this stage, but the arguments below with $s$ and $t$ interchanged
remain valid.)  Thus either
\begin{itemize}
\item[(iii)] there exist $m,n>n_{s,t}$ such that $sd_mp=td_np$; or
\item[(iv)] there exist $m>n_{s,t}$ and $r''\in D^*$ such that $sd_mp=tr''p$.
\end{itemize}

If (iii) holds, then by \cite[Lemma 6.28]{HS}, $sd_m=td_n$, contradicting the choice of $n_{s,t}$.
So (iv) holds.  But $r''\in V_{sd_m,t}$ so $tr''p\notin U_{sd_m}$, a contradiction.
\end{proof}

We now present several results about the dynamical systems considered in Section 3.

\begin{lemma}\label{AcapbetaSp} Let $S$ be a semigroup and let $p$ be 
a right cancelable element of $\beta S$. 
Then for any clopen subset
$E$ of $\beta Sp$, there is some $A\subseteq S$ such that 
$E=\overline A\cap \beta Sp$. 
\end{lemma}

\begin{proof} Let $E$ be a clopen subset of $\beta Sp$.
Let ${\cal D}=\{\overline D\cap\beta Sp:D\subseteq S
\hbox{ and }\overline D\cap\beta Sp\subseteq E\}$.  Since $\{\overline D\cap\beta Sp:D\subseteq S\}$
is a basis for the topology of $\beta Sp$ and $E$ is open in $\beta Sp$, we have that
$E=\bigcup{\cal D}$. Since $E$ is compact, pick finite ${\cal F}\subseteq{\cal P}(S)$
such that $E=\bigcup_{D\in{\cal F}}(\overline D\cap\beta Sp)$ and let
$A=\bigcup{\cal F}$.\end{proof}

\begin{theorem}\label{rtcancimpsur} Let $S$ be a semigroup. Let 
$(Y,\langle T_s\rangle_{s\in S})$
be the dynamical system of Lemma {\rm \ref{stotwodyn}} determined by $R=S$. Let $p\in\beta S$.
If $p$ is right cancelable in $\beta S$, then $T_p:Y\to Y$ is surjective.
\end{theorem}

\begin{proof} Note that since $\rho_p:\beta S\to\beta Sp$ is injective and
takes closed sets to closed sets, it is a homeomorphism. 

To see that $T_p$ is surjective, let $z\in Y$, let
$B=\{s\in S:z(s)=1\}$, and let $E=\rho_p[\,\overline B\,]$.
Then $E$ is clopen in $\beta Sp$ so  by Lemma \ref{AcapbetaSp} pick  $A\subseteq S$ such that
$E=\overline A\cap \beta Sp$. Let $x$ be the characteristic function of
$A$ in $Y$.  
We claim that $T_p(x)=z$.  For this, it suffices that for each
$s\in S$, $\{t\in S:T_t(x)(s)=z(s)\}\in p$. So let 
$s\in S$. Note that $\{t\in S:T_t(x)(s)=1\}=\{t\in S:x(st)=1\}=s^{-1}A$.
Also $s^{-1}A\in p$ if and only if $s\in \rho_p^{-1}[\,\overline A\cap \beta Sp]$
so $s\in B$ if and only if $s^{-1}A\in p$.

If $z(s)=1$, then $s\in B$ so $s^{-1}A\in p$ so
$\{t\in S:T_t(x)(s)=z(s)\}\in p$. If $z(s)=0$, then $s\notin B$ so $s^{-1}A\notin p$ so
$\{t\in S:T_t(x)(s)=z(s)\}\in p$. 
\end{proof}

Notice that the hypotheses of the following corollary hold if $S$ has any
right cancelable element.

\begin{corollary}\label{rtcanciffpsurY} Let $S$ be a semigroup. Let 
$(Y,\langle T_s\rangle_{s\in S})$
be the dynamical system of Lemma {\rm \ref{stotwodyn}} determined by $R=S$. Let $p\in\beta S$.
Assume that for whenever $q$ and $r$ are distinct elements of $\beta S$, there
exists $s\in S$ such that $sq\neq sr$.
Then $T_p:Y\to Y$ is surjective if and only if
$p$ is right cancelable in $\beta S$.\end{corollary}

\begin{proof} The necessity is Theorem \ref{rtcancimpsur}.

So assume that $T_p$ is surjective and suppose that we have distinct
$q$ and $r$ in $\beta S$ such that $qp=rp$.  We claim that 
$T_q=T_r$.  To see this, let $x\in Y$ be given. Pick $z\in Y$ such that
$T_p(z)=x$. Then $T_q(x)=T_q\big(T_p(z)\big)=T_{qp}(z)=T_{rp}(z)=
T_r\big(T_p(z)\big)=T_r(x)$.  

Pick $s\in S$ such that $sq\neq sr$,
pick $A\in sq\setminus sr$, and let $x$ be the characteristic function
of $A$ in $Y$.  Then $A\subseteq\{t\in S:T_t(x)(s)=1\}$ so $T_q(x)(s)=1$
and $S\setminus A\subseteq\{t\in S:T_t(x)(s)=0\}$ so $T_r(x)(s)=1$.
\end{proof}

\begin{theorem}\label{rtcanciffpsurX} Let $S$ be a semigroup and
let $Q=S\cup\{e\}$ where $e$ is an identity adjoined to $S$. Let
$(X,\langle T_s\rangle_{s\in S})$
be the dynamical system of Lemma {\rm \ref{stotwodyn}} determined by $R=Q$ and
let $p\in\beta S$. Then $T_p:X\to X$ is surjective if and only if
$p$ is right cancelable in $\beta Q$.
\end{theorem}

\begin{proof} Sufficiency. Note that $\rho_p:\beta S\to\beta Sp$ is a homeomorphism.
Note also that $p\notin \beta Sp$. (If we had $p=qp$ for some $q\in\beta S$, then
we would have $ep=qp$.) Let $x\in X$ and let $B=\{s\in S:x(s)=1\}$. By Lemma
\ref{AcapbetaSp}, pick $A\subseteq S$ such that $\rho_p[\,\overline B\,]=
\overline A\cap \beta Sp$. Pick $P\in p$ such that $\overline P\cap \beta Sp=\emp$.
If $x(e)=1$, let $D=A\setminus P$. If $x(e)=0$, let $D=A\cup B$.  Let
$z$ be the characteristic function of $D$ in $X$. 

We claim that $T_p(z)=x$. As in the proof of Theorem \ref{rtcancimpsur}, we
see that for $s\in S$, $T_p(z)(s)=x(s)$.
Regardless of the value of $x(e)$, we have that 
$P\subseteq\{s\in S:T_s(z)(e)=x(e)\}$, so $T_p(z)(e)=x(e)$. 

Necessity. Suppose that $T_p$ is surjective and we have $q\neq r$ in $\beta Q$ such that
$qp=rp$. Assume first
that $e\in \{q,r\}$, so without loss of generality, $q=e$. Let
$x$ be the characteristic function of $S$ in $X$ and pick
$z\in X$ such that $T_p(z)=x$. Then
$0=x(e)=T_p(z)(e)=T_{rp}(z)(e)=T_r\big(T_p(z)\big)(e)
=T_r(x)(e)=1$, a contradiction.

So we can assume that $q$ and $r$ are in $\beta S$.  Pick
$A\in q\setminus r$ and let $A$ be the characteristic function
of $A$ in $X$. Pick $z\in X$ such that $T_p(z)=x$. Then 
$0=T_r(x)(e)=T_{rp}(z)(e)=T_{qp}(z)(e)=T_q\big(T_p(z)\big)(e)
=T_q(x)(e)=1$, a contradiction. 
\end{proof}

\begin{theorem}\label{Singroup}
Let $S$ be a countable semigroup which can be embedded in a group and assume that
$S$ can be enumerated as $\langle s_t\rangle_{t=0}^\infty$ so that if $u,v\in S$, $i,j\in\omega$ with
$i<j$, and $s_iu=s_jv$, then $s_0s_i^{-1}s_j\in S$. 
Let $(Y,\langle T_s\rangle_{s\in S})$
be the dynamical system of Lemma {\rm \ref{stotwodyn}} determined by $R=S$ and let $p\in\beta S$. 
The $T_p$ is surjective if and only if there exists $x\in Y$ such that $T_p(x)$ is the characteristic
function of $\{s_0\}$ in $Y$.\end{theorem}

\begin{proof} The necessity is trivial. 
Assume that we have $x\in Y$ such that $T_p(x)$ is the characteristic
function of $\{s_0\}$ in $Y$.  For $m\in\ben$, let
$D_m=\{s_0s_i^{-1}s_j:i,j\in\ohat{m}\,,\, i<j\hbox{, and }
s_0s_i^{-1}s_j\in S\}$ and note that $s_0\notin D_m$. For each
$m\in\ben$, let 
$$\textstyle B_m=\{s\in S:T_s(x)\in\pi_{s_0}^{-1}[\{1\}]
\cap\bigcap_{i=1}^m\pi_{s_i}^{-1}[\{0\}]\cap\bigcap_{r\in D_m}\pi_r^{-1}[\{0\}]\}\,,$$
and note that $B_m\in p$.  We claim that if $m,k\in\ben$, $u\in B_m$, 
$v\in B_k$, $i\in\ohat{m}$,  $j\in\ohat{k}$, and $s_iu=s_jv$, then
$i=j$.  Suppose instead we have such $m,k,u,v,i,j$ with $i\neq j$ and
assume without loss of generality that $i<j$.  Then $u=s_i^{-1}s_jv$.
By assumption $s_0s_i^{-1}s_j\in S$ so $s_0s_i^{-1}s_j\in D_k$.
Since $u\in B_m$, $1=T_u(x)(s_0)=x(s_0u)$.  Since $v\in B_k$ and
$s_0s_i^{-1}s_j\in D_k$, $0=T_v(x)(s_0s_i^{-1}s_j)=x(s_0s_i^{-1}s_jv)$,
a contradiction.

Now to see that $T_p$ is surjective, let $y\in Y$ be given.  Define $w\in Y$ as 
follows. If $m\in\ben$, $u\in B_m$, and $i\in\ohat{m}$, then 
$w(s_iu)=y(s_i)$. For $s\in S$ which is not of the form 
$s_iu$ for some $m\in\ben$, $u\in B_m$, and $i\in\ohat{m}$, define
$w(s)$ at will.  To see that $T_p(w)=y$, let $U$ be a neighborhood of 
$y$. Pick $m\in\ben$ such that $\bigcap_{i=0}^m\pi_i^{-1}[\{y(s_i)\}]\subseteq U$.
Then $B_m\subseteq U$.
\end{proof}

The following is an immediate corollary of Theorem \ref{Singroup}.

\begin{corollary}\label{ingroup} Let $S$ be a countable group with identity $e$, let $(Y,\langle T_s\rangle_{s\in S})$
be the dynamical system of Lemma {\rm \ref{stotwodyn}} determined by $R=S$, and let $p\in\beta S$. 
The following statements are equivalent.
\begin{itemize}
\item[(a)] $T_p$ is surjective.
\item[(b)] For each $s\in S$, there exists $x\in Y$ such that $T_p(x)$ is the characteristic
function of $\{s\}$.
\item[(c)] There exists $x\in Y$ such that $T_p(x)$ is the characteristic
function of $\{e\}$.
\end{itemize}\end{corollary}

Notice that the hypotheses of the following theorem hold if $S$ is
very weakly left cancellative and right cancellative. If $\kappa$ is regular,
the assumption that for any subset $D$ of $S$ with fewer than 
$\kappa$ members,  $|\{e\in S:(\exists s\in D)(\exists t\in D\setminus\{s\})(se=te)\}|<\kappa$
can be replaced by the assumption that for all distinct $s$ and $t$ in $S$,
$|\{e\in S:se=te\}|<\kappa$.

\begin{theorem}\label{vwlckappa}
Let $S$ be a semigroup with $|S|=\kappa\geq\omega$ 
which is very weakly left cancellative and
has the property that for any subset $D$ of $S$ with fewer than 
$\kappa$ members,  $|\{e\in S:(\exists s\in D)(\exists t\in D\setminus\{s\})(se=te)\}|<\kappa$.
Let $(Y,\langle T_s\rangle_{s\in S})$
be the dynamical system of Lemma {\rm \ref{stotwodyn}} determined by $R=S$.
There is a dense open subset $W$ of $U_\kappa(S)$ such that 
for every $p\in W$, $p$ is right cancelable in $\beta S$ and
$T_p:Y\to Y$ is surjective.\end{theorem}

\begin{proof} We show that for any $C\in [S]^\kappa$, there exists
$B\in [C]^\kappa$ such that for every $p\in \overline B\cap U_\kappa(S)$, $p$ is right cancelable in $\beta S$ and
$T_p:Y\to Y$ is surjective.

Enumerate $S$ as $\langle s_\gamma\rangle_{\gamma<\kappa}$.
Choose $t_0\in C$.  Let $0<\alpha<\kappa$ and assume that
we have chosen $\langle t_\delta\rangle_{\delta<\alpha}$ in $C$ satisfying
the following inductive hypotheses:
\begin{itemize}
\item[(1)] If $\gamma<\delta$, then $t_\gamma\neq t_\delta$.
\item[(2)] If $\gamma<\delta$, $\mu<\beta\leq \delta$, and $\mu\neq\gamma$,
then $s_\gamma t_\delta\neq s_\mu t_\beta$.
\end{itemize}

The hypotheses are satisfied for $\delta=0$.
Let $E=\{e\in S:(\exists \mu<\beta\leq \alpha)(s_\mu e=s_\beta e)\}$.
For $\mu<\beta<\alpha$ and $\gamma<\alpha$ let
$A_{\gamma,\mu,\beta}=\{t\in S:s_\gamma t=s_\mu t_\beta\}$.
Then each $A_{\gamma,\mu,\beta}$ is a left solution set.
Pick $$\textstyle t_\alpha\in C\setminus(\{t_\gamma:\gamma<\alpha\}\cup 
E\cup\{\bigcup_{\gamma<\alpha}\bigcup_{\beta<\alpha}\bigcup_{\mu<\beta}A_{\gamma,\mu,\beta})\,.$$

Hypothesis (1) is trivially satisfied and if $\mu<\beta<\alpha$ and
$\gamma<\alpha$, then $t_\alpha\notin A_{\gamma,\mu,\beta}$ so
$s_\gamma t_\alpha\neq s_\mu t_\beta$. If
$\mu<\beta=\alpha$ and $\gamma<\alpha$, then $t_\alpha\notin E$ so
$s_\gamma t_\alpha\neq s_\mu t_\beta$.

Let $B=\{t_\alpha:\alpha<\kappa\}$ and let $p\in\overline B\cap U_\kappa(S)$.
To see that $p$ is right cancelable in $\beta S$, let $q\neq r\in \beta S$ and suppose
that $qp=rp$.  Pick subsets $C$ and $D$ of $S$ such that $C\cap D=\emp$ and 
$C\in q$ and $D\in r$. Then $H=\{s_\gamma t_\alpha:\gamma<\alpha\hbox{ and }s_\gamma\in C\}
\in qp$.  (To see this, let $s_\gamma\in C$. Then $\{t_\alpha:\gamma<\alpha<\kappa\}\subseteq s_{\gamma}^{-1}H$.)
Similarly, $\{s_\mu t_\beta:\mu<\beta\hbox{ and }s_\mu\in D\}
\in rp$.  Since these sets are disjoint by hypothesis (2), we have a contradiction.

The fact that $T_p$ is surjective follows from Theorem \ref{rtcancimpsur}.
\end{proof}

\begin{lemma}\label{chianotrec} Let $S$ be a cancellative semigroup, let 
$a\in S$, and let $(Y,\langle T_s\rangle_{s\in S})$
be the dynamical system of Lemma {\rm \ref{stotwodyn}} determined by $R=S$.
If $x$ is the characteristic function of $\{a\}$ in $Y$, then
$x$ is not a recurrent point.
\end{lemma}

\begin{proof} We claim that $|\{s\in S:T_s(x)(a)=1\}|\leq 1$.
Indeed, if $x(as)=1$, then $as=a$ so by left cancellation, $s$ 
is a left identity for $S$ and then by right cancellation, $s$ 
is a two sided identity for $S$.\end{proof}

We have seen that $U(x)$ is always a left ideal of $\beta S$ and that $NS$ is a right
ideal of $\beta S$ provided it is nonempty.

\begin{theorem}\label{NSnotleft} Let $S$ be a countable cancellative semigroup.
Let $(Y,\langle T_s\rangle_{s\in S})$
be the dynamical system of Lemma {\rm \ref{stotwodyn}} determined by $R=S$.
Then $NS_Y$ is not a left ideal of $\beta S$.
\end{theorem}

\begin{proof} By \cite[Corollary 6.33]{HS} pick an idempotent $p\in \beta S\setminus K(\beta S)$.
By Theorem \ref{rpcancel} pick $r\in\beta S$ such that $rp$ is right cancelable in $\beta S$.
By Lemma \ref{chianotrec} and Theorem \ref{idemNS}, $p\in NS$ and by 
Theorem \ref{rtcancimpsur}, $rp\notin NS$.
\end{proof}

If $S$ is commutative, then by \cite[Exercise 4.4.9]{HS} and Theorem \ref{NSright},
if $NS\neq\emp$, then $\cl NS$ is a two sided ideal of $\beta S$.  The following 
theorem shows that this may fail if $S$ is not commutative.

\begin{theorem}\label{freesg} Let $S$ be the free semigroup on the alphabet $\{a,b\}$ (where $a\neq b$). 
Let $(Y,\langle T_s\rangle_{s\in S})$
be the dynamical system of Lemma {\rm \ref{stotwodyn}} determined by $R=S$.
Then $NS\neq\emp$ and $\cl NS$ is not a left ideal of $\beta S$.
\end{theorem}

\begin{proof} Let $p$ be an idempotent in $\beta S$ with $\{a^n:n\in\ben\}\in p$.
By Lemma \ref{chianotrec} and Theorem \ref{idemNS}, $p\in NS$.  We will show that
$bp\notin\cl NS$.  Let $B=\{ba^n:n\in\ben\}$. Then $B\in bp$. We shall show that
$\overline B\cap NS=\emp$.  So let $q\in \overline B$.
Let $s_0=a$ and let $\langle s_n\rangle_{n=1}^\infty$ enumerate $S\setminus\{a\}$ so that
if the length of $s_i$ is less than the length of $s_j$, then $i<j$.
By Theorem \ref{Singroup}, to see that $T_q$ is surjective, it suffices to show that
there is some $x\in Y$ such that $T_q(x)$ is the characteristic function of $\{a\}$.

Let $x$ be the characteristic function of $\{aba^n:n\in\ben\}$ in $Y$.
Let $U$ be a neighborhood of $\cchi_{\{a\}}$ and pick $F\in \pf(S\setminus\{a\})$
such that $\pi_a^{-1}[\{1\}]\cap\bigcap_{y\in F}\pi_{y}^{-1}[\{0\}]\subseteq U$.
It suffices to show that $B\subseteq\{w\in S:T_w(x)\in \pi_a^{-1}[\{1\}]\cap\bigcap_{y\in F}\pi_{y}^{-1}[\{0\}]\}$.
So let $ba^n\in B$. Then $T_{ba^n}(x)(a)=x(aba^n)=1$ and for 
$y\in F$, $T_{ba^n}(x)(y)=x(yba^n)=0$.
\end{proof}

We remark that Theorem \ref{freesg} remains
valid if $S$ is the free semigroup on a countably infinite alphabet.

\bibliographystyle{plain}

\end{document}